\documentclass[journal]{IEEEtran4PSCC}

\usepackage{cite}

\ifCLASSINFOpdf
\usepackage[pdftex]{graphicx}
\else
\usepackage[dvips]{graphicx}
\fi
\usepackage{color}

\usepackage[cmex10]{amsmath}
\usepackage{amsthm}
\usepackage{amssymb}
\usepackage{subcaption}
\usepackage{tabularx}
\usepackage{bbm}
\usepackage{multirow}

\newtheorem{thm}{Theorem}
 \theoremstyle{definition}
 \newtheorem{defn}{Definition}

\newcommand{\pmin}{p^{\text{min}}}
\newcommand{\pmax}{p^{\text{max}}}
\newcommand{\fmin}{f^{\text{min}}}
\newcommand{\fmax}{f^{\text{max}}}
\newcommand{\argmin}{\text{argmin}}
\newcommand{\ac}{\mathcal{A}}

\usepackage{url}

\renewcommand{\baselinestretch}{0.9675}

\begin{document}
	\title{Statistical Learning for DC Optimal Power Flow}
	
	\author{Yeesian Ng\IEEEauthorrefmark{1}, Sidhant Misra\IEEEauthorrefmark{2}, Line A. Roald\IEEEauthorrefmark{2}, Scott Backhaus\IEEEauthorrefmark{2}
	\thanks{\IEEEauthorrefmark{1}: Operations Research Center, Massachusets Institute of Technology, Cambridge, United States. Email: yeesian@mit.edu}
	\thanks{\IEEEauthorrefmark{2}: Los Alamos National Laboratory, Los Alamos, NM, USA. Email: \{sidhant, roald, backhaus\}@lanl.gov}
	\vspace*{-1em}}
		
	\maketitle
	
	\begin{abstract}
	The optimal power flow problem plays an important role in the market clearing and operation of electric power systems. However, with increasing uncertainty from renewable energy operation, the optimal operating point of the system changes more significantly in real-time.
	In this paper, we aim at developing control policies that are able to track the optimal set-point with high probability. The approach is based on the observation that the OPF solution corresponding to a certain uncertainty realization is a basic feasible solution, which provides an affine control policy. The optimality of this \emph{basis policy} is restricted to uncertainty realizations that share the same set of active constraints. We propose an \emph{ensemble control policy} that combines several basis policies to improve performance. Although the number of possible bases is exponential in the size of the system, we show that only a few of them are relevant to system operation. We adopt a statistical learning approach to learn these important bases, and provide theoretical results that validate our observations. For most systems, we observe that efficient ensemble policies constructed using as few as ten bases, are able to obtain optimal solutions with high probability.

 	\end{abstract}

	\begin{IEEEkeywords}
		Optimal Power Flow, Uncertainty, Statistical Learning, Ensemble Methods
	\end{IEEEkeywords}
	
	\section{Introduction}
	
    The DC Optimal Power Flow (OPF) is a widely used optimization problem in power systems operation and market clearing. The OPF attempts to find the most economic dispatch of generators that satisfy the demand and technical constraints in the system. With increasing levels uncertainty due to higher renewable penetration, the generation dispatch requires larger and more frequent adjustments in real-time in order to maintain power balance and feasibility at all times. In the literature, it is often assumed that the generators respond to such short term uncertainty according to an affine control policy \cite{borkowska, vrakopoulou2012, roald2017corrective}, which is a good representation of the Automatic Generation Control (AGC) commonly used in system operation.
    
    While affine control policies perform well when the uncertainty is limited, it is advantageous to consider more general control policies in the case of larger deviations. For example, it was shown that a piece-wise affine (PWA) policy that models the activation of tertiary reserves in response to large uncertainty realizations is more economic \cite{roald2015optimal}. 
    In this paper, we investigate the characteristics of a control policy which is able to track the optimal solution to the real-time OPF, which is a Linear Program (LP) whose parameters are dictated by the realization of the uncertainty. By tracking the manifold of OPF solutions,
    this policy can provide significant improvements in terms of economy and ensuring system feasibility, particularly when the uncertainty is large.
    
    It is known in the Model Predictive Control (MPC) literature that the optimal policy is PWA,     
    and can be computed explicitly by recursively  partitioning the space of parameters into the so-called ``critical regions" for each affine piece \cite{bemporad2002model,borrelli2003geometric,alessio2009survey}. 
    This observation was used in \cite{Vrakopoulou2017-ek} to show that the optimal policy to the OPF is piecewise affine, by utilizing the Multi-parametric Toolbox (MPT \cite{herceg2013multi}). However, the result was restricted to small system sizes, owing to the fact that the MPT methods do not scale very well to large dimensions. 
    Even approximate techniques for the computation of PWA policies  \cite{parisini1995receding,bemporad2003suboptimal,johansen2003approximate,de2004robust,christophersen2007controller,canale2009set,bemporad2011ultra} are still not practical for instances of realistic sizes.
    
    In this paper, we propose a novel and scalable approach to solve the real-time OPF that leverages statistical learning to learn certain important features of the OPF solution, and construct a PWA policy, referred to as the \emph{ensemble control policy}, based on these features. 
    The method exploits the fact that the optimal solution to a LP is basic feasible, i.e., it is determined by a set of active constraints that are satisfied with equality. The ensemble control policy uses a collection of bases (independent vectors in ${\Re}^n$) where each basis corresponds to a set of active constraints. The complexity of the ensemble policy is governed by the number of bases used in its construction.
    
	Although the number of possible bases is exponentially large in the size of the system, typically only a few are relevant within the operational time frame.
	Our statistical learning approach is able to effectively identify and select the most important bases in cases where the number of bases is indeed small. Furthermore, we present a statistical results which allows us to identify and diagnose exceptional systems for which there are a large number of important bases and our method is likely to produce sub-optimal results.
	
    This paper combines the above observations to construct a control policy which tracks the solution of the real-time OPF by (i) using off-line learning to identify the important bases, and (ii) implementing a real-time system control using an efficient ensemble policy constructed based on a small number of bases. This two-step procedure allows us to construct policies with high probability of providing optimal solutions, even for large scale systems.

    Although the main focus of this paper is the real-time OPF, the learning based framework generalizes to other applications such as learning and predicting the behavior of locational marginal prices (LMPs) \cite{Geng2015-af}. Furhter, the method can be extended to non-linear problems such as the real-time AC-OPF, that cannot be addressed by existing methods.
    
    The remainder of the paper is organized as follows. Section \ref{sec:theory} describes the OPF problem formulation. Section \ref{sec:theoryLP} discusses the construction of the ensemble policies based on LP theory, while Section \ref{sec:statlearning} describes the learning procedure. Extensive numerical results for a range of test cases are provided in Section \ref{sec:numerical-results}, while Section \ref{sec:conclusions} summarizes and concludes the paper.
    \section{Problem Formulation} \label{sec:theory}
	
	The electric transmission network is assumed to be a graph $\mathcal{G}=(\mathcal{V,\mathcal{E}})$, where $\mathcal{V}$ denotes the nodes of the graph corresponding to the buses with $|\mathcal{V}| = v$, and the edges $\mathcal{E}$ denote the transmission lines with $|\mathcal{E}|=m$. The total number of generators is given by $n$.
	
	We begin by stating the OPF problem for a given uncertainty realization:
	\begin{subequations} \label{eq:opf}
	\begin{align}
        \rho^*(\omega)\in\underset{p}{\text{argmin}}\ &c^\top p \\
    	\text{s.t.}\ &\pmin\leq p\leq\pmax \\
    	&\fmin\leq M(Hp+\mu+\omega-d)\leq\fmax \\
	    &e^\top p = e^\top(d-\mu-\omega)
	\end{align}
	\end{subequations}
	
    Here, the generated active power at each generator $p\in\Re^n$ are the decision variables, with a linear cost coefficient $c$.
	The parameters $\mu\in\Re^{v}$ represent the forecasted production of non-dispatchable active power (e.g. from wind or solar PV), $\omega\in\Re^{v}$ is the uncertain deviation from the forecasted value and $d\in\Re^{v}$ is the vector of demands. The vectors $\pmin,\pmax\in\Re^n$ correspond to the minimum and maximum power generation limits, and $\fmin,\fmax\in\Re^m$ encode the minimum and maximum transmission flow limits. Further, $e$ is the vector of ones, $H\in\Re^{v\times n}$ is the matrix mapping the power from each generator to their corresponding bus, and $M\in\Re^{m\times v}$ is the matrix of power transfer distribution factors \cite{christie2000transmission}. The set of all uncertainty realizations $\omega$ is denoted by $\Omega$.
	
	Let $\mathcal{P}^*(\omega)$ denote the set of minimizers of the optimization problem in \eqref{eq:opf}. At this point, it is prudent to exclude uncertainty realizations $\omega$ from the set $\Omega$ for which there exists no feasible generation dispatch in \eqref{eq:opf}. Therefore, in the following analysis, we restrict ourselves to the set $\Omega^R:=\{\omega:\mathcal{P}^*(\omega)\neq\emptyset\}$ of \textit{recoverable} scenarios, i.e., the scenarios for which a feasible solution can be found.
	
	For each $\omega \in \Omega^R$, the set $\mathcal{P}^*(\omega)$ can be interpreted as the set of the best possible generation dispatch actions for the uncertainty realization $\omega$. This motivates the definition of an \emph{optimal control policy} below.
	\begin{defn} \label{def:optimal_policy}
	A \emph{control policy} $\rho:\Omega\rightarrow\Re^n$ is a mapping that adjusts the generation in response to uncertainty. It is said to be \textit{optimal} if $\rho(\omega) \in \mathcal{P}^*(\omega)$ for all $\omega \in \Omega^R$.
	\end{defn}
	By definition, $\rho^*(\cdot)$ as defined in \eqref{eq:opf} is an optimal control policy. We note that the problem in \eqref{eq:opf} belongs to the class of \emph{Parametric Linear Programming} problems, with the uncertainty $\omega$ serving the role of the parameter. For such problems, there always exists an optimal control policy that is a continuous and PWA function of $\omega$ over the entire domain $\Omega^R$ \cite{walkup1969lifting,alessio2009survey}.

	\section{Polyhedral Theory and Construction of Ensemble Policies} \label{sec:theoryLP}
	Although modern LP solvers are efficient, computing $\rho^*(\cdot)$ can still require significant on-line computation. Moreover, it does not provide insight into the dependence of the decisions $\rho(\omega)$ on the forecast error $\omega$. In this section, we review results from polyhedral theory to gain insight into the structure of the optimal PWA policies. We then describe a procedure to construct a particular kind of PWA control policies, termed \emph{Ensemble Policies}, based on this insight.
	
	\subsection{Polyhedral Theory of Linear Programming}
	First, we observe that for each uncertainty realization $\omega$, the OPF in \eqref{eq:opf} is a LP whose feasible set is the polyhedron given by
    \begin{align}
    	\mathcal{P}(\omega)=\{p\in\Re^n:\ &\pmin\leq p\leq\pmax,  \nonumber \\
	    &\fmin\leq M(Hp+\mu+\omega-d)\leq\fmax, \nonumber \\
	    &e^\top p = e^\top(d-\mu-\omega)\}. \label{eq:polyhedron}
	\end{align}
	It is a well-known property of LPs that the optimal solution to a non-degenerate instance of \eqref{eq:opf} lies at a corner of the above polyhedron (see Theorem~\ref{thm:bfs}). 
	We recall some definitions related to this fact.
	
	\begin{defn}
	For a generation dispatch vector $p\in\Re^n$,
	\begin{itemize}
		\item[(a)] $p$ is a \textit{basic} solution if it satisfies the power balance constraint $e^\top(p+\mu+\omega-d)=0$ and $(n-1)$ other linearly independent constraints that are active at $p$,
		\item[(b)] $p$ is a \textit{basic feasible} solution (BFS) if it is a basic solution that satisfies all of the constraints, i.e. $p\in \mathcal{P}(\omega)$.
	\end{itemize}
	\end{defn}
	
	\begin{thm} \cite{bertsimas1997introduction} \label{thm:bfs}
	    For any $\omega\in\Omega^R$, there exists a basic feasible solution $p^*\in P^*(\omega)$. \label{thm:BFS}
	\end{thm}

	Note that if the problem is degenerate, several corners and the faces between them might be optimal.
	
	For the OPF problem, each BFS corresponds to a set of active line and generator constraints. 
	We next show how each such solution corresponds to an optimal basis, which can be used to define an affine control policy.

    \subsection{From Basic Feasible Solutions to Affine Control Policies}
    The optimization problem has $n$ decision variables $p\in\Re^n$. At a BFS $p^*\in\mathcal{P}(\omega)$ for a given $\omega\in\Omega^R$, exactly $n$ linearly independent constraints in \eqref{eq:polyhedron} are satisfied with equality, such that the generator output is uniquely determined. We will use this fact to construct an affine policy for $p$ as a function of $\omega$. 
    
    Since the power balance constraint $e^\top p = e^\top(d-\mu-\omega)$ is always satisfied with equality, there must be $n-1$ remaining rows of $A$ that correspond to equality constraints.
    For ease of exposition, we re-write the constraints describing the polyhedron $\mathcal{P}(\omega)$ in \eqref{eq:polyhedron} compactly in the following form
    \begin{align} \label{eq:polyhedron_matrix}
        \mathcal{P}(\omega) = \{p: \ Ap \leq b+C\omega,  \quad 
                        e^\top p = e^\top(d-\mu-\omega)    \},
    \end{align}
    where 
    \begin{align}
        A \!=\! &\left[ \begin{array}{r}
            \! I \\
            \! -I \\
            \! MH \\
            \! -MH
        \end{array}
        \right]\!\!\in\Re^{2(n+m)\times n}, \quad
        C \!=\! \left[ \begin{array}{r}
            \! 0 \\
            \! 0 \\
            \! -M \\
            \! M
        \end{array}
        \right]\!\!\in\Re^{2(n+m)\times v}, \nonumber  \\[+2pt]
         b \!=\! &\left[ \begin{array}{l}
            \! ~~\pmax \\
            \! -\pmin \\
            \! ~~\fmax - M(\mu-d) \\
            \! -\fmax + M(\mu-d)
        \end{array}
        \right]\!\!\in\Re^{2(n+m)}. \nonumber
    \end{align}
    
    Let $\ac = \{i_1, i_2, \ldots, i_{n-1}\}$ denote the indices of $n-1$ linearly independent rows of $A$. The basis matrix $B \in {\Re}^{n \times n}$ is then defined as
    \begin{align} \label{eq:basis-matrix}
        B = \left[ \begin{array}{c}
             A_{\ac}  \\
             e^\top 
        \end{array}
        \right],
    \end{align}
    where $A_{\ac}$ is the submatrix of $A$ formed by the rows in $\ac$. 
    As we find the set $\ac$ easier to work with, but the notion of a basis $B$ more conceptually useful, we will switch between them interchangeably through \eqref{eq:basis-matrix}. 
    By the definition of a basic solution, we have that
    \begin{align} \label{eq:affine_form}
        p^* = B^{-1} \left[ \begin{array}{c}
             b_{\ac} + C_{\ac} \omega  \\
             e^\top(d-\mu-\omega) 
        \end{array} \right] =: \rho^{\ac}(\omega).
    \end{align}
    Namely, we can view $p^*$ as the output of an affine policy $\rho^{\ac}(\cdot)$ evaluated at $\omega$.
    Note that the affine policy \eqref{eq:affine_form} can be evaluated for any uncertainty realizations $\omega\in\Omega$, leading to an adjusted set of generation outputs $p$. This observation leads to the following definition.
    
    \begin{defn}
	For any set $\ac$ of $n-1$ linearly independent active rows of $A$, we define the \textit{basis policy} $\rho^{\ac}(\cdot)$ as in \eqref{eq:affine_form}.
	\end{defn}
	
	\subsection{Analysis of the Affine Basis Policies}

    Each basis policy $\rho^{\ac}(\cdot)$ corresponds to a set of active constraints, which are sometimes refferred to as system pattern regions \cite{Geng2015-af}.
    This allows us to interpret their behavior. 
    For example, based on the characterization of the basis in \eqref{eq:basis-matrix} and \eqref{eq:affine_form}, we can classify the generators in the system into the following categories:
	\begin{itemize}
		\item[(i)] $I^{\text{UB}} = \{i: p_i=\pmax_i\}$, the indices of generators that should be set to their maximum generation limit,
		\item[(ii)] $I^{\text{LB}} = \{i: p_i=\pmin_i\}$, the indices of generators that should be set to their minimum generation limit, and
		\item[(iii)] $I^{\text{vary}} = [n] \setminus (I^{\text{UB}}\cup I^{\text{LB}})$, the indices of generators that should vary linearly as a function of $\omega$ by solving the remaining system of $n-1-|I^{\text{UB}}|-|I^{\text{LB}}|$ linear equations specified by the rows of $\ac$ indexed by $I^{\text{vary}}$.
	\end{itemize}

    The affine basis policy will provide optimal solutions for all $\omega$ which share the same set of active constraints specified by the basis matrix $B$. However, it might not perform very well for other realizations $\omega\in\Omega^R$, as the solution it provides might either violate some of the constraints in $\mathcal{P}(\omega)$ or be suboptimal in terms of generation cost.
    
    To assess the performance of any basis policy $\rho^{\ac}(\omega)$ more precisely, we define the sets
	\begin{align*}
	&\Omega_{\ac} := \Omega_{\rho^\ac} = \{\omega: \rho^{\ac}(\omega)\in \mathcal{P}(\omega) \} \\
	&\Omega_{\ac}^* := \Omega_{\rho^\ac}^* = \{\omega: \rho^{\ac}(\omega)\in \mathcal{P}^*(\omega) \}
	\end{align*}
	corresponding to the set of scenarios for which $\rho^{\ac}(\cdot)$ provides a feasible and optimal power generation solution, respectively. It is not hard to verify that
	\begin{equation*}
	\Omega_{\ac}^*\subseteq\Omega_{\ac}\subseteq\Omega^R,
	\end{equation*}
	which says that the set of scenarios for which a given basis is optimal is contained within the set of scenarios for which the basis is feasible. The scenario set $\Omega_{\ac}^*$ corresponds to the set of scenarios that share the same set of active constraints and hence the same optimal basis. 

	\subsection{Ensemble Policies}
    Having established that each uncertainty realizations corresponds has a corresponding affine policy that can be derived from the optimal basis, we look at an approach to form a control policy based on an ensemble of basis policies. This \emph{ensemble policy} is optimal (respectively feasible) over a wider range of scenarios than any of its constituent basis policies.
    
	We denote the total number of bases by $b$, and observe that
	  \begin{equation} \label{eq:num_bases}
        b \leq \binom{2(n+m)}{n}.
    \end{equation}
    For an arbitrary ordering $\mathcal{A}^{(1)},\dots,\mathcal{A}^{(b)}$ of the bases, we obtain a corresponding sequence of basis policies $\rho^{(1)},\dots,\rho^{(b)}$, where $\rho^{(i)}:=\rho^{A^{(i)}}$ is the policy induced by the $i$-th basis. For any given subset $I\subseteq [b] = \{1,2,\ldots,b\}$ of the indices, we construct the ensemble policy as described below.
	
    \vspace*{0.5em}
    \noindent\rule{\linewidth}{1pt} \\[0pt]
    \noindent \textbf{Ensemble Control Policy:} $I \subseteq [b]$  \\[-4pt]
    \noindent\rule{\linewidth}{1pt}
    \noindent  The ensemble control policy for a given $I$ is given by 
    \begin{align} \label{eq:ensemble_policy}
        \rho^I(\omega) :=\underset{\rho^{(i)}(\omega): i\in I}{\argmin}\ &c^\top \rho^{(i)}(\omega) \\
    	\text{s.t.}\ &\rho^{(i)}(\omega)\in \mathcal{P}(\omega) \nonumber
    \end{align}
    \vspace*{0.5em}
    \noindent\rule{\linewidth}{1pt}
    
	The feasible domain of the ensemble policy in \eqref{eq:ensemble_policy} is given by the union of the feasible domains of the basis policies $\Omega^I:=\cup_{i\in I}\Omega_{\rho^{(i)}}$. If there are more than one feasible policy for the given scenario, the ensemble policy selects the solution that minimizes the cost.
	The selection of the lowest cost, feasible solution can be carried out by a simple exhaustive evaluation of the cost and feasibility of each basis policy $\rho^{(i)}$ for $i \in I$.
	
	If we choose $I = [b]$ and form an ensemble based on all of the bases of the LP, then by Theorem~\ref{thm:BFS} the resulting policy $\rho^{[b]} = \rho^*$ is an optimal control policy.
	However, since $b$ is typically exponentially large in the size of the network, it is computationally prohibitive to use the full ensemble policy $\rho^{[b]}(\cdot)$. To address this, we provide a framework that allows us to trade-off between (i) the computational complexity and (ii) the feasibility and optimality of the solutions provided by the resulting policy. The key component of this framework is to constructing ensemble policies of increasing complexity given by $\rho^{[1]}, \rho^{[2]}, \ldots, \rho^{[b]}$, where $\rho^{[i]}$ refers to the ensemble policy formed by the set of basis policies $\{\rho^{(1)}, \ldots, \rho^{(i)}\}$. The number of bases $i$ we choose to include allows us to tradeoff computation and performance.
	
	In this approach, the ordering of the bases (i.e., \emph{which} $i$ bases we choose to include first) is critical in determining how many bases we need to achieve our desired level of trade-off.
	Intuitively, we wish to place the \emph{important} bases that have the highest probability of providing an optimal solution for a given realization $\omega$, earlier in the sequence, and the ones less relevant for operational practice (e.g. bases containing constraints that are typically never encountered to be tight during operations) later in the sequence.
    To determine this ordering, we adopt a statistical learning based approach to identify the bases that are most relevant to the scenarios that would arise in practice.
	
	\section{Using Statistical Learning to Identify Important Bases}
	The importance of a basis can be quantified by the probability that it is optimal for the OPF. Let $\mathbb{P}_{\omega}$ denote the probability distribution of $\omega$. This induces a probability distribution over the set of all bases $[b]$ that describes the probability of each of them to be optimal for the OPF problem \eqref{eq:opf}
    \begin{defn} \label{def:pi_def}
        For any basis $\ac$ we denote by $\pi(\ac)$ the probability that it is optimal for the OPF problem\footnote{Note that, due to degeneracy, there may be multiple bases that may be optimal for a given uncertainty realization $\omega$ in \eqref{eq:pi_def} . To make the definition consistent, any tie-breaking rule which ensures that the same basis will always be defined as optimal for a given $\omega$ is sufficient.}.
        \begin{align} 
            \pi(\ac) = \mathbb{P}_{\omega}(\Omega_{\mathcal{A}}^*) = \mathbb{P}_{\omega}(\ac \mbox{ is optimal for } \eqref{eq:opf}). \label{eq:pi_def}
        \end{align}
        For any ensemble $I \subseteq [b]$ of the set of all bases, we define
        \begin{align} 
            \pi(I) = \sum_{i \in I} \pi(\ac^{(i)}).  \label{eq:coverage_def}
        \end{align}
    \end{defn}
    In this section, we present a method to identify the most important (i.e., most probable) bases based on statistical learning. The key idea is to observe the relevant basis and their probabilities by evaluating \eqref{eq:opf} for a large number of uncertainty samples $M$. The approach has two main components.  
    First, in order to ensure that the ensemble policy performs well, we would like to guarantee that we are able to discover a set of bases which captures a large fraction of the probability mass within $\Omega^R$ using a limited number of samples $M$. To achieve this, we establish a criterion which provides such a statistical guarantee if it is satisfied.
    Second, we suggest to form a reduced ensemble policy (including only a sub-set of the discovered bases) by ordering the bases based on their empirical probability of occurrence.
    
    \subsection{Statistical Results for Unobserved Bases}
    \label{sec:statlearning}
   In this section, we provide a criterion to check whether the set of observed bases covers a significant amount of the probability mass, and show that the quantity that provide such guarantees is related to the so-called \emph{rate of discovery} of previously unobserved bases.
    
    \begin{defn} \label{def:unobserved_set}
    Let $\omega_1, \ldots, \omega_M$ be $M$ i.i.d. samples drawn from the uncertainty distribution and $\ac_1, \ldots, \ac_M$ denote the optimal basis corresponding to the OPF solution for each $\omega_i$. We call $\mathcal{O}_M = \cup_{i=1}^M \{\ac_i\}$ the set of \emph{observed bases}, and $\mathcal{U}_M = \mathcal{B} \setminus \mathcal{O}_M$ the set of \emph{unobserved bases}.
    \end{defn}
    
    We now define the rate of discovery, which quantifies the fraction of samples that correspond to observing a basis that was not observed before.
    
    \begin{defn} \label{def:rod}
    Let $W$ be a positive integer denoting the window size. Let $\omega_1, \ldots, \omega_{M+W}$ be $M+W$ i.i.d. samples drawn from the uncertainty distribution and let $\ac_i$ denote the optimal basis corresponding to the OPF solution for $\omega_i$. We denote by $X_i$ the random variable that encodes whether a new basis was observed in the $(M+i)^{th}$ sample, i.e.
     \begin{align}
	        X_i = \begin{cases}
	                    1, \ \mbox{if} \ \ac_{M+i} \notin \{\ac_1\} \cup \ldots \cup \{\ac_{M}\}, \\
	                    0, \ \mbox{otherwise}.
	               \end{cases}
	    \end{align}
	    Then the \emph{rate of discovery} over the window of size $W$ is given by $\mathcal{R}_W$ and is defined as
	    \begin{align} \label{eq:R_def}
	        \mathcal{R}_W  = \frac{1}{W} \sum_{i=1}^{W} X_i,
	    \end{align}
    \end{defn}
    
    The following theorem guarantees that the rate of discovery is unlikely to fall below a certain threshold if the mass of the unobserved set is large.
    
    \begin{thm} \label{thm:coverage_test}
       Let the unobserved set $\mathcal{U}_M$ and the rate of discovery $\mathcal{R}_W$ be defined as in Definition~\ref{def:unobserved_set} and \ref{def:rod}. Let $\epsilon$ and $\delta$ be given positive numbers, corresponding to the probability mass of the unobserved set and the confidence of our experiment, respectively. Then
       \begin{align}
           \mathbb{P}\left(\mathcal{R}_W < \frac{\epsilon}{2} \mid \pi(\mathcal{U}_M) > \epsilon \right) < \delta,
       \end{align}
       provided that the window size $W$ satisfies
       \begin{align}
           W > \frac{8}{\epsilon} \log \frac{1}{\delta}. \label{eq:window_size}
       \end{align}
    \end{thm}
    
    Essentially, Theorem~\ref{thm:coverage_test} tells us that it is unlikely to observe an average rate of discovery $< \epsilon/2$ over the window size $W$ if the probability mass of the unobserved set of bases $\pi(\mathcal{U}_M)$ has fallen below a predefined threshold $\epsilon$. By choosing parameters $M,~W,~\epsilon$ and $\delta$, we can check whether the rate of discovery at the end of our experiment has become sufficiently low to guarantee that we did not miss too much probability mass.
    
    \vspace*{0.5em}
    \noindent\rule{\linewidth}{1pt} \\[0pt]
    \noindent \textbf{Probability Mass of Unobserved Set: $\epsilon, \delta, M, W$}  \\[-4pt]
    \noindent\rule{\linewidth}{1pt}
    \noindent Draw $M+W$ i.i.d. samples from the uncertainty distribution, with $W > \frac{8}{\epsilon} \log \frac{1}{\delta}$ and solve the OPF for each scenario to obtain the associated optimal bases $\ac_1, \ldots, \ac_{M+W}$. 
    If $\mathcal{R}_W > \epsilon/2$, declare that the learning procedure is inconclusive, otherwise declare success.
    \vspace*{0.5em}
    \noindent\rule{\linewidth}{1pt}

	\subsection{Ordering of Bases by Empirical Probability} \label{sec:learning-procedure}
	Equipped with Theorem~\ref{thm:coverage_test} to provide guarantees for the probability mass contained in the observed bases, we use the empirically observed probability of the bases to provide an ordering that identifies an ensemble $I$ which is optimal for the OPF problem with high probability.
	
    \vspace*{0.5em}
    \noindent\rule{\linewidth}{1pt} \\[0pt]
    \noindent \textbf{Learn Important Bases: $\epsilon, \delta, K$}  \\[-4pt]
    \noindent\rule{\linewidth}{1pt}
    
    \noindent \textbf{\textit{Step 1:}} Order the bases observed in the first $M$ samples according to their empirical probability of observation given by $\hat{\pi}(\ac) = \frac{1}{M} \sum_{i=1}^{M} \mathbf{1}_{\ac_i = \ac}$. Let $\ac_{i_1}, \ldots, \ac_{i_K}$
    be the first $K$ bases in the ordering.
    
    \noindent \textbf{\textit{Step 2:}} Form an ensemble of size $K$ using the bases $\ac_{i_1},\ldots,\ac_{i_K}$ to construct the ensemble policy in \eqref{eq:ensemble_policy}.
    \vspace*{0.5em}
    \noindent\rule{\linewidth}{1pt}

	\section{Numerical Results} \label{sec:numerical-results}
	
	We demonstrate the efficacy of our statistical learning procedure in providing feasible and optimal solutions with a high probability, by running extensive simulations across a range of different systems, and report the results with accompanying discussions in the next two sections.
	
	To evaluate the performance, we ran the learning procedure (described in Section \ref{sec:learning-procedure}) across a variety of networks \cite{li2010small,lesieutre2011examining,grigg1999ieee,birchfield2017grid,price2011reduced} from the IEEE PES PGLib-OPF v17.08 benchmark library \cite{pglib_opf}, performing the evaluations in Julia v0.6 \cite{julia}, using JuMP v0.17 \cite{DunningHuchetteLubin2017} and PowerModels.jl v0.5 \cite{power_models}. We report general results for 15 different test cases, varying in size from 3 to 1951 buses, and provide an overview of the characteristics of the considered systems in Table \ref{tab:characteristics}.
    For each system, we assume that the loads are uncertain, and follow a multivariate normal distribution. Hence, $\omega$ is defined as random vector of independent, zero mean variables and standard deviations $\sigma$, which are defined as a fraction of the load (referred to as $\sigma$-scaling in the following). We assume zero correlation between loads. Note that the assumption of a normal distribution is not necessary, but is chosen to enable a test set-up which is easy to replicate. 	
    
    	\begin{table}[b!]
\centering
\begin{tabular}{|l|c|c|c|c|c|c|}
\hline
\multirow{2}{*}{Case Name} & \multirow{2}{*}{Buses} & \multirow{2}{*}{Lines} & Gene- & Cons- & Infeasible \\
          &       &       & rators & straints & Scenarios \\
\hline
case3\_lmbd & 3 & 3 & 3 & 13 & 0\\
case5\_pjm & 5 & 6 & 5 & 23 & 0\\
case14\_ieee & 14 & 20 & 5 & 51 & 0\\
case24\_ieee\_rts & 24 & 38 & 33 & 143 & 0\\
case30\_ieee & 30 & 41 & 6 & 95 & 8\\
case39\_epri & 39 & 46 & 10 & 113 & 0\\
case57\_ieee & 57 & 80 & 7 & 175 & 0\\
case73\_ieee\_rts & 73 & 120 & 99 & 439 & 0\\
case118\_ieee & 118 & 186 & 54 & 481 & 0\\
case162\_ieee\_dtc & 162 & 284 & 12 & 593 & 88\\
case200\_pserc & 200 & 245 & 38 & 567 & 23\\
case240\_pserc & 240 & 448 & 143 & 1183 & 15\\
case300\_ieee & 300 & 411 & 69 & 961 & 0\\
case1888\_rte & 1888 & 2531 & 290 & 5643 & 0\\
case1951\_rte & 1951 & 2596 & 366 & 5925 & 0\\
\hline
    \end{tabular}
    \caption{\small Summary of the network characteristics  ($\sigma$-scaling=0.03), as well as the number of scenarios that are infeasible for the OPF in each system, considering 10'000 samples.}
    \label{tab:characteristics}
\end{table}

    \subsection{Learning Important Bases}
	First, we take a look at the properties of the learning process. For an increasing number of samples $M$, we first assess the number of unique bases observed for different levels of volatility and systems of varying size, and then evaluate the cumulative proportion of all scenarios covered by the set of bases $\mathcal{O}_M$ that has been observed. At the end of the experiment, we check whether our criterion based on the average rate of innovation guarantees that the set of discovered bases contain a minimum level of probability mass.  
	
	For our assessment, we first generated $M=5000$ samples from the distribution on the buses assuming a $\sigma$-scaling = 0.03. We then solve the OPF for each one of them, keeping track of the basis corresponding to each scenario, as well as the number of scenarios that lead to OPF infeasibility (i.e., the number of non-recoverable scenarios). 
    The total number of infeasible samples is listed along with the system characteristics in Table \ref{tab:characteristics}, and is reasonably small for each system.
	In addition to the results for each PGLib system, we perform a more in-depth analysis of the IEEE 300 bus system, by considering 10'000 samples and different $\sigma$-scalings $\sigma=\{0.01,\,0.02,\,0.03,\,0.04,\,0.05\}$.

	\subsubsection{Number of Unique Bases} We first assess the number of unique bases observed for an increasing number of samples. The number of unique bases that have been discovered after a given number of samples have been drawn are shown in Figure \ref{fig:nbases300vsnscenarios300} for the IEEE 300 bus system with varying levels of uncertainty and in Table \ref{tab:n-unique-bases} for the different PGLib cases. 
	
	The results for the IEEE 300 bus system corroborates our intuition that the number of unique optimal bases is correlated with the volatility of the load variation. At lower levels of volatility (when $\sigma$-scaling is lower than $0.04$), the number of new unique bases quickly stabilizes at a low number. At higher levels of volatility, the number of unique bases grew more rapidly and never stabilized. At these higher levels of volatility, the system was infeasible more frequently.
	
	The results for the PGLib cases show that for most systems, the number of unique bases stabilizes at a low value after 5000 samples. Smaller systems such as the case24\_ieee\_rts observing a similar number of unique bases as the case1951\_rte, showing that system size is not a good indicator of the number of unique bases. Instead, the important factor is whether the system has a frequently changing set of active constraints. A special case for this is the case240\_pserc, for which we observe a large and non-stabilizing number of bases.
	
	\begin{figure}[t!]
	\centering
	\includegraphics[width=0.8\linewidth]{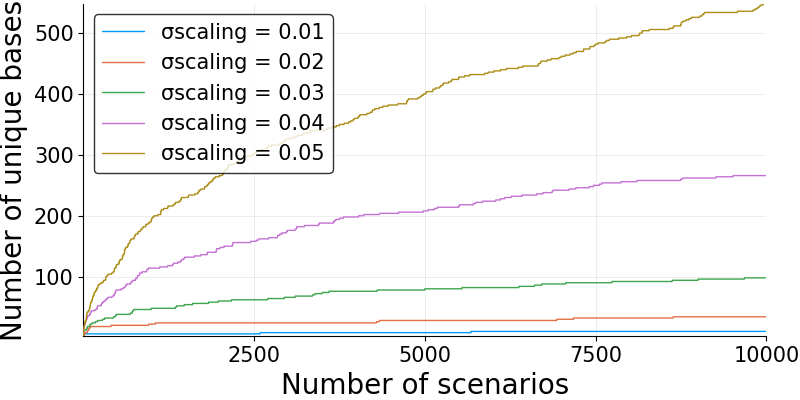}
	\caption{\small Number of unique optimal bases identified for a given number of samples for the IEEE 300 bus system with varying $\sigma$-scaling.}
	\label{fig:nbases300vsnscenarios300}
	\vspace{+8pt}
	\includegraphics[width=0.8\linewidth]{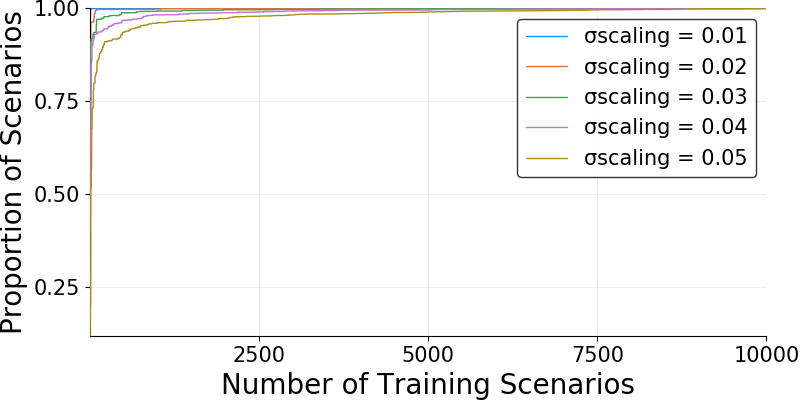}
	\caption{\small Proportion of all 10'000 scenarios covered by the already identified bases.}
	\label{fig:proportion300vsnscenarios300}
	\end{figure}
	
	\subsubsection{Discovery of Important Bases as a Function of Samples}
	Intuitively, we expect the more important bases would have a higher probability of being identified by any given scenario, and hence have a high chance of being identified early on in the training process. To assess whether our intuition holds, for each new scenario, we collect the bases that have been identified up til then, and plot the proportion of all scenarios that correspond to any of the already identified bases. The results are shown graphically for the IEEE 300 bus system in Figure \ref{fig:proportion300vsnscenarios300}, and for the PGLib cases in Table \ref{tab:my_label}. The results are encouraging in suggesting that the important bases are indeed found early in the training phase, rather than being sporadically identified late into the training phase. 

    \begin{table}[t!]
    \centering
    \begin{tabular}{|l|c|c|c|c|c|c|}
    \hline
    \# of samples & 100 & 200 & 500 & 1000 & 2500 & 5000 \\
    \hline
    case3\_lmbd & 1 & 1 & 1 & 1 & 1 & 1\\
    case5\_pjm & 1 & 1 & 1 & 1 & 1 & 1\\
    case14\_ieee & 1 & 1 & 1 & 1 & 1 & 1\\
    case24\_ieee\_rts & 5 & 5 & 5 & 5 & 10 & 10\\
    case30\_ieee & 1 & 1 & 1 & 1 & 1 & 1\\
    case39\_epri & 2 & 2 & 2 & 2 & 2 & 2\\
    case57\_ieee & 2 & 2 & 3 & 3 & 3 & 3\\
    case73\_ieee\_rts & 13 & 14 & 14 & 21 & 27 & 27\\
    case118\_ieee & 2 & 2 & 2 & 2 & 2 & 2\\
    case162\_ieee\_dtc & 7 & 8 & 9 & 9 & 11 & 11\\
    case200\_pserc & 43 & 58 & 88 & 109 & 162 & 180\\
    case240\_pserc & 69 & 114 & 239 & 391 & 707 & 1114\\
    case300\_ieee & 9 & 11 & 18 & 22 & 26 & 37\\
    case1888\_rte & 3 & 3 & 3 & 3 & 3 & 3\\
    case1951\_rte & 5 & 6 & 7 & 7 & 9 & 10\\
    \hline
    \end{tabular}
    \caption{\small Number of unique optimal bases identified for a given number of samples at $\sigma$-scaling=0.03.}
    \label{tab:n-unique-bases}
    \end{table}
	
	\subsubsection{Probabilistic Guarantees for the Discovered Bases}
	Finally, we apply the results in Theorem \ref{thm:coverage_test} to assess whether we are able to guarantee that the discovered bases cover a sufficient probability mass. For our calculations, we assume that the probability mass of the undiscovered set should not exceed $\epsilon=2$\% with confidence $1-\delta=90$\%. This corresponds to checking whether the average rate of discovery $\mathcal{R}_W$ for the last $W=921$ samples given by \eqref{eq:window_size}, has fallen below $\epsilon/2=1$\%. 
	Checking this criterion for the IEEE 300 bus system and the PGLib cases, we observe that the criterion is mostly satisfied with rates of discovery $<0.3$\%, indicating that we can say with confidence that the most important bases have been discovered. If we choose to include all observed bases in our ensemble policy, we will obtain optimal solutions with probability $1-\epsilon = 98$\%. 
	The only systems for which the criterion does not hold are the IEEE 300 bus cases with $\sigma$-scaling$={0.04,\,0.05}$, case200\_pserc and case240\_pserc. For the former three cases, the rate of discovery is $<2.6$\%, while the last case has a rate of discovery of about $15$\%. These corresponds to the test systems for which the number of bases had grown very large at the end of the experiment, indicating that there are many relevant bases which occur with low probability and that many of those have not yet been observed.
	
	\subsection{Performance of Ensemble Policy}
	We investigate how the number of bases considered in the ensemble policy affects its performance, in terms of providing both feasible and optimal results. We rank the bases based on their empirical probabilities and construct ensemble policies with an increasing number of bases. We then evaluate their performance in an out-of-sample test based on 5000 new samples generated by the same multivariate normal distributions described above, using a $\sigma$-scaling=0.03 for the PGLib cases and varying $\sigma$-scalings $\sigma=\{0.01,\,0.02,\,0.03,\,0.04,\,0.05\}$ for the IEEE 300 bus system.
	
	\begin{table}[t!]
    \centering
    \begin{tabular}{|l|c|c|c|c|c|c|c|}
    \hline
    \# of samples & 100 & 200 & 500 & 1000 & 2500 & 5000 \\
    \hline
    case3\_lmbd & 1.000 & 1.000 & 1.000 & 1.000 & 1.000 & 1.000\\
    case5\_pjm & 1.000 & 1.000 & 1.000 & 1.000 & 1.000 & 1.000\\
    case14\_ieee & 1.000 & 1.000 & 1.000 & 1.000 & 1.000 & 1.000\\
    case24\_ieee\_rts & 0.613 & 0.613 & 0.613 & 0.613 & 1.000 & 1.000\\
    case30\_ieee & 0.998 & 0.998 & 0.998 & 0.998 & 0.998 & 0.998\\
    case39\_epri & 1.000 & 1.000 & 1.000 & 1.000 & 1.000 & 1.000\\
    case57\_ieee & 1.000 & 1.000 & 1.000 & 1.000 & 1.000 & 1.000\\
    case73\_ieee\_rts & 0.000 & 0.000 & 0.000 & 0.942 & 1.000 & 1.000\\
    case118\_ieee & 1.000 & 1.000 & 1.000 & 1.000 & 1.000 & 1.000\\
    case162\_ieee\_dtc & 0.978 & 0.981 & 0.981 & 0.981 & 0.982 & 0.982\\
    case200\_pserc & 0.689 & 0.792 & 0.900 & 0.938 & 0.975 & 0.988\\
    case240\_pserc & 0.494 & 0.574 & 0.664 & 0.747 & 0.809 & 0.847\\
    case300\_ieee & 0.961 & 0.973 & 0.983 & 0.991 & 0.993 & 0.998\\
    case1888\_rte & 1.000 & 1.000 & 1.000 & 1.000 & 1.000 & 1.000\\
    case1951\_rte & 0.994 & 0.995 & 0.998 & 0.998 & 1.000 & 1.000\\
    \hline
    \end{tabular}
    \caption{\small Cumulative proportion of scenarios covered by the already observed basis for a given number of samples at $\sigma$-scaling=0.03, based on an out-of-sample test with 5000 scenarios drawn from the multivariate normal distribution.}
    \label{tab:my_label}
    \end{table}
	
	\begin{figure}[b!]
		\centering
		\includegraphics[width=0.8\linewidth]{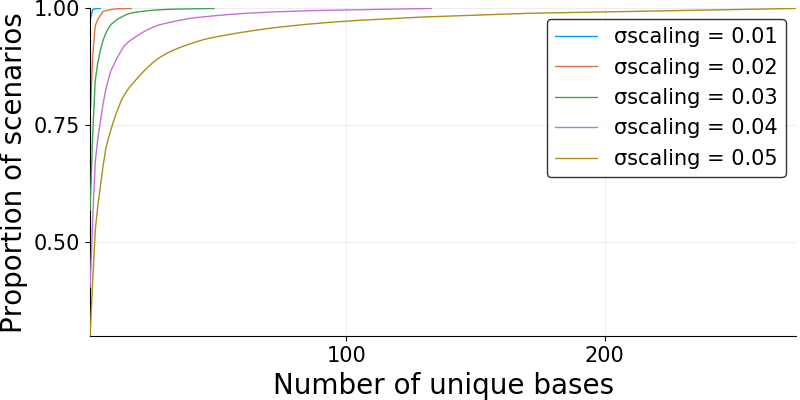}
		\caption{\small Proportion of all scenarios covered by a given number of bases (included in order of maximum probability).}
		\label{fig:proportion300vsnbases300}
	\end{figure}
	
	    \begin{table}[]
    \centering
    \begin{tabular}{|l||c c c||c c c|}
    \multicolumn{1}{l}{} & \multicolumn{3}{c}{\textbf{Optimal solutions}} & \multicolumn{3}{c}{\textbf{Feasible solutions}} \\
    \hline                   
    \# of bases        & 5 & 10 & 100                 & 5 & 10 & 100 \\
    \hline
    case3\_lmbd        & 1.000 & 1.000 & 1.000        & 1.000 & 1.000 & 1.000\\
    case5\_pjm         & 1.000 & 1.000 & 1.000        & 1.000 & 1.000 & 1.000\\
    case14\_ieee       & 1.000 & 1.000 & 1.000        & 1.000 & 1.000 & 1.000\\
    case24\_ieee\_rts  & 0.932 & 1.000 & 1.000        & 0.968 & 1.000 & 1.000\\
    case30\_ieee        & 1.000 & 1.000 & 1.000        & 1.000 & 1.000 & 1.000\\
    case39\_epri        & 1.000 & 1.000 & 1.000        & 1.000 & 1.000 & 1.000\\
    case57\_ieee        & 1.000 & 1.000 & 1.000        & 1.000 & 1.000 & 1.000\\
    case73\_ieee\_rts   & 0.900 & 0.981 & 1.000        & 0.991 & 0.991 & 1.000\\
    case118\_ieee       & 1.000 & 1.000 & 1.000     & 1.000 & 1.000 & 1.000\\
    case162\_ieee\_dtc  & 0.983 & 0.999 & 1.000     & 0.983 & 0.999 & 1.000\\
    case200\_pserc      & 0.345 & 0.476 & 0.949     & 0.623 & 1.000 & 1.000\\
    case240\_pserc      & 0.270 & 0.355 & 0.663     & 0.270 & 0.355 & 0.664\\
    case300\_ieee       & 0.903 & 0.972 & 1.000     & 0.903 & 0.972 & 1.000\\
    case1888\_rte       & 1.000 & 1.000 & 1.000     & 1.000 & 1.000 & 1.000 \\
    case1951\_rte       & 0.994 & 1.000 & 1.000     & 0.994 & 1.000 & 1.000\\
    \hline
    \end{tabular}
    \caption{\small Proportion of scenarios for which the ensemble policy returns an optimal (left) or feasible (right) solution, based on the given number of bases at $\sigma$-scaling=0.03.}
    \label{tab:optimal_feasible}
    \end{table}
	
	Figure \ref{fig:proportion300vsnbases300} shows the results for the IEEE 300 bus system. The quicker the curve saturates to a value close to $1$, the better the performance is, since it indicates that with a small ensemble of basis policies (and hence fewer computations), one can obtain optimal (and feasible) solutions with a high probability. Unsurprisingly, the number of bases required to obtain optimal solutions for a given probability increases as the $\sigma$-scaling increases. 
	Similar results for the PGLib cases are shown on the left of Table \ref{tab:optimal_feasible}. For most systems, an ensemble policy with 10 bases is already sufficient to obtain optimal solution for a high proportion $>0.99$ of the scenarios. The two systems for which we were not able to guarantee that the probability mass of the undiscovered bases was less than $\epsilon=0.02$, case200\_pserc and case240\_pserc, the proportion of optimal scenarios is much lower, even with 100 bases. 
	
    While the probability of obtaining optimal solutions is an important performance criterion for the ensemble policy, system operators also care about feasibility of system operation. To assess the feasibility performance, the right part of Table \ref{tab:optimal_feasible} shows the proportion of feasible scenarios for an ensemble policy with different number of bases, for each PGLib case. For most cases, the proportion of optimal scenarios is very close to the proportion of feasible scenarios, indicating that the ensemble policy either returns optimal or infeasible solutions. One exception is the case200\_pserc, which has a much higher level of feasibility than optimality.

	\section{Conclusions}
	\label{sec:conclusions}
	We develop ensemble control policies to solve the real-time DC-OPF by combining the affine basis polices corresponding to the basic feasible solutions of the linear program. Although the computational complexity of the ensemble policy is dictated by the number of constituent bases which can be exponentially many in the size of the system, we show that only a few of those are relevant to power systems operations under uncertainty. We adopt a statistical learning approach to learn the important bases from experiments with a limited number of samples, and provide theoretical results that justify the learning procedure. We show that for almost all systems, regardless of its size, an ensemble policy with only $10$ bases.
    
    Future work will involve extending the theoretical results, as well as applications to a wider range of power systems problems, including the non-linear AC OPF problem

	\appendix
	
	\begin{proof}[Proof of Theorem~\ref{thm:coverage_test}]
	 For any $0<\gamma<\epsilon$ we have
	 \begin{align}
	     \mathbb{P}(\mathcal{R}_W &< \gamma \mid \pi(\mathcal{U}_M) > \epsilon)     \nonumber \\ 
	     &= \sum_{u:\pi(u)>\epsilon}    \mathbb{P}\left(\mathcal{R}_W < \gamma \mid \mathcal{U}_M = u \right) \mathbb{P}(\mathcal{U}_M=u) \nonumber \\
	     &= \sum_{u:\pi(u)>\epsilon} \mathbb{P}\left(\mathcal{R}_W < \gamma \mid \mathcal{U}_M = u \right) \mathbb{P}(\mathcal{U}_M = u)
	     \label{eq:partial_proof}
	 \end{align}
	 Conditioned on $\mathcal{U}_M=u$ the random variables $X_i$ for are Bernoulli random variables with probabiliy of success $p > \pi(u) > \epsilon$. Then by the Chernoff inequality, we have
	 \begin{align*}
	     \mathbb{P}\left(\mathcal{R}_W < \gamma \mid \mathcal{U}_M=u\right) &< e^-{\frac{(\sum p_i - W \gamma)^2}{2\sum p_i}} \\
	     &\leq e^-{\frac{W(\epsilon-\gamma)^2}{2\epsilon}}.
	 \end{align*}
	 Using the above in \eqref{eq:partial_proof} and using $\gamma = \epsilon/2$ we get
	 \begin{align*}
	    \mathbb{P}\left(\mathcal{R}_W < \gamma \mid \pi(\mathcal{U}_M) > \epsilon \right) &\leq \sum_{u:\pi(u)>\epsilon} 
	    e^{-\frac{W\epsilon}{8}} \mathbb{P}(\mathcal{U}_M = u) \\
	    &\leq e^{-\frac{W\epsilon}{8}} < \delta,
	 \end{align*}
	 where the last inequality follows from the assumption on the window size $W$ in \eqref{eq:window_size}.
	\end{proof}
	
	\bibliographystyle{IEEEtran}
	\bibliography{pscc2018.bib}{}
	
\end{document}